\documentclass[10pt,a4paper]{article}
\usepackage[margin=2cm]{geometry}
\usepackage{amsmath,amsthm,amsfonts,amssymb,amscd,cite,mathrsfs,graphicx,url}
\usepackage{latexsym}
\usepackage{enumitem}
\usepackage{thmtools}
\usepackage{mathabx}\changenotsign
\usepackage{dsfont}
\usepackage[abbrev,msc-links]{amsrefs}

\usepackage[usenames, dvipsnames]{color}
\definecolor{mygray}{gray}{0.7}

\usepackage[table]{xcolor}
\usepackage{wasysym,upgreek}
\usepackage{mathdots}
\usepackage[enableskew]{youngtab}
\usepackage{tikz}
\usetikzlibrary{tikzmark,decorations.pathreplacing}
\usetikzlibrary{arrows,matrix}
\usetikzlibrary{positioning}
\usetikzlibrary{decorations}
\usepackage[all,cmtip]{xy}

\usepackage{titlesec}
\titleformat{\section}
{\normalfont\fontsize{16}{20}\bfseries}{\thesection}{1em.}{}
\definecolor{amber}{rgb}{1.0,0.75,0.0}
\definecolor{applegreen}{rgb}{0.55,0.71,0.0}
\definecolor{byzantium}{rgb}{0.44, 0.16, 0.39}
\definecolor{cadmiumorange}{rgb}{0.93, 0.53, 0.18}
\definecolor{darkcyan}{rgb}{0.0, 0.55, 0.55}
\newtheorem{lemma}{Lemma}[section]

\newtheorem{theorem}{Theorem}[section]
\newtheorem{remark}{Remark}[section]

\newtheorem{cor}{Corollary}[section]

\newtheorem{prop}{Proposition}[section]

\newtheorem{claim}{Claim}[section]
\newtheorem{prob}{Problem}[section]

\let\oldbibliography\thebibliography
\renewcommand{\thebibliography}[1]{%
  \oldbibliography{#1}%
  \setlength{\itemsep}{-2pt}%
}
\let\E=\EE
\def\rm{\text{$\mathbb{RM}$}}
\def\PP{{\mathds P}}

\newcommand{\cM}{\mathcal{M}}
\def\rd{\text{$\mathbb{RD}$}}
\def\rw{\text{$\mathbb{RW}$}}

\AtBeginDocument{%
   \def\MR#1{}
}

\begin{document}
\baselineskip=0.20in

\makebox[\textwidth]{%
\hglue-15pt
\begin{minipage}{0.6cm}	
%\vfill
\vskip9pt
\end{minipage} \vspace{-\parskip}
%\begin{minipage}[t]{11cm}
%\includegraphics[width=2.8cm]{ECA.pdf}
%\footnotesize{}% {\bf Enumerative Combinatorics and Applications} \\ \underline{ecajournal.haifa.ac.il}}
%\end{minipage}
\hfill
\begin{minipage}[t]{5.4cm}
\normalsize{} % {\it ECA}  {\bf X:X} (xxxx) Article \#SxRx
\end{minipage}}
\vskip36pt

\begin{center}
{\large \bf Largest bipartite sub-matchings of a random ordered matching or a problem with socks}\\[10pt]

Andrzej Dudek$^\dag$, Jaros\l aw Grytczuk$^\ddag$ and Andrzej Ruci\'nski$^\S$\\[20pt]

\footnotesize {\it $^\dag$Department of Mathematics, Western Michigan University, Kalamazoo, MI, USA\\
Email: andrzej.dudek@wmich.edu}\\[10pt]

\footnotesize {\it $^\ddag$Faculty of Mathematics and Information Science, Warsaw University of Technology, Warsaw, Poland\\
Email: jaroslaw.grytczuk@pw.edu.pl}\\[10pt]

\footnotesize {\it $^\S$Department of Discrete Mathematics, Adam Mickiewicz University, Pozna\'n, Poland\\
Email: rucinski@amu.edu.pl}\\[20pt]

%\noindent {\footnotesize {\bf Received}: April XX, xxxx},
%{\footnotesize {\bf Accepted}: August XX, xxxx}, {\footnotesize {\bf Published}: August XX, xxxx}\\

%\noindent The authors: Released under the CC BY-ND license (International 4.0)
\end{center}

\baselineskip=0.30in

\normalsize

\noindent
{\sc Abstract:} Let $M$ be an \emph{ordered matching} of size $n$, that is, a partition of the set $[2n]$ into 2-element subsets. The \emph{sock number} of $M$ is the maximum size of a sub-matching of $M$ in which all left-ends of the edges  precede all the right-ends (such matchings are also called bipartite). The name of this parameter comes from an amusing ``real-life'' problem posed by Bosek, concerning an \emph{on-line} pairing of randomly picked socks from a drying machine.
	Answering one of Bosek's questions we prove that the sock number of a random matching of size $n$ is asymptotically equal to $n/2$. Moreover, we prove that the expected average number of socks waiting for their match during the whole process is equal to $\frac{2n+1}{6}$. Analogous results are obtained if socks come not in pairs, but in sets of size $r\geq 2$, which corresponds to a similar problem for random ordered $r$-matchings. We also attempt to enumerate matchings with a given sock number.
\bigskip



\section{Introduction}

Bart\l{}omiej Bosek \cite{Bosek} has asked the following ``practical'' questions. There is a laundry sack with $n$ different pairs of (clean) socks.
You pull randomly socks from the sack, one by one, and place them on the floor, unless the sock in your hand matches one already on the floor in which case you bundle the pair up and put it into a drawer. What is the largest number of socks on the floor at any given time? What is the average number of socks on the floor throughout the whole process? Of course, these are questions about random variables in a suitable probability space.

The linear order in which the socks are pulled from the sack can be identified with one of $(2n)!/2^n$ permutations with $n$ pairwise repetitions. However, as it really does not matter which socks are pulled and when, but only whether they form a pair with one already on the floor, we can truncate the space down to just all $(2n)!/(2^nn!)$ (ordered) $n$-element matchings of the set $[2n]=\{1,2,\dots, 2n\}$.

For instance,  matching $\{1,5\}, \{2,3\},\{4,6\}$ corresponds to the order of socks in which the first and the fifth socks pulled are matched and so are the second and the third, as well as the fourth and the sixth. Here the situation on the floor goes through the following six stages: $1, 12,1$ (because the second and the third sock are put together in the drawer), $14, 4, \emptyset$. So, in this case the answers to the two above questions are $2$ and $7/6$, respectively. If, instead, we are looking at matching $\{1,5\}, \{2,4\},\{3,6\}$, the answers are $3$ and $3/2$.

There is a convenient way to represent ordered matchings in terms of \emph{Gauss words}, i.e., words in which every letter of an $n$-element alphabet appears exactly twice and  words obtained by permuting their letters are identified (e.g., $ABBA=BAAB$).  Clearly, we get such a word from an ordered matching if the ends of each edge are replaced with a pair of identical letters, different form all other pairs of letters.
%More precisely, we fix a linear order of the alphabet and assign the $i$-th  letter to the $i$-th edge (counting, say, from the left).
For example, the two instances above can be encoded as $ABBCAC$ and $ABCBAC$, respectively.

Formally, let $M$ be a matching on $[2n]$. For each $k\in[2n]$, let $x_k:=x_k(M)$ be the number of edges of $M$ with one endpoint in $[k]$ and the other in $[2n]\setminus[k]$. Note that $x_k$ represents exactly the number of socks on the floor after $k$ socks have been pulled from the sack. For that reason we call the sequence $(x_k)_{k=1}^{2n}$ the \emph{sockuence} of $M$ (the term coined in by Martin Milani\v c).

Let $\rm(n)$ be a random matching on $[2n]$, that is, a matching  picked uniformly at random out of the set of all
$\frac{(2n)!}{(2)^n\, n!}$ matchings on the set $[2n]$. Thus, defining random variables $X_k=x_k(\rm(n))$,  the two questions of Bosek ask about distributions of the random variables $Y=\max_kX_k$ -- the largest value of the sockuence of $\rm(n)$ and $\bar X=\sum_kX_k/(2n)$ -- the average value of the sockuence, respectively.
In this paper we determine both random variables asymptotically, the latter being a simpler task.

\begin{theorem}\label{thm1:r=2}
We have $\E\bar X=\frac{2n+1}{6}$ and, for every $\omega(n)\to\infty$, a.a.s.
\[
| \bar X- n/3|\le \omega(n)\sqrt n.
\]
\end{theorem}

As for the former task, note that the random variable $Y$ equals the size of the largest \emph{bipartite} sub-matching of $\rm(n)$, that is, a sub-matching consisting of edges whose all left-ends precede all right-ends. Equivalently, a bipartite matching is one with the interval chromatic number equal to 2 (see, e.g., \cite{FJKMV} for definition). In the word notation, a \emph{bipartite} ordered matching is just a Gauss word whose second half  is a permutation of the first half, as it happens, for instance, in  the word $ABCDBDAC$. There is yet another alternative definition of bipartite matchings in terms of forbidden sub-structures. Indeed, there are three types, or patterns, in which two edges may intertwine:  an \emph{alignment} $AABB$, a \emph{nesting} $ABBA$, or a \emph{crossing} $ABAB$.
Then a bipartite matching is precisely one with no alignments.

It has been known (see \cite{JSW} for alignments,  \cite{BaikRains} for nestings and crossings, see also \cite{DGR-match}) that for each of the three patterns, a.a.s.\  the largest size of  a sub-matching of $\rm(n)$  with all pairs of edges forming that pattern has size $\Theta(\sqrt n)$. It follows that a.a.s.\ $Y=\Omega(\sqrt n)$.
Here we show that, in fact, a.a.s.\ $Y$ is a linear function of $n$.

\begin{theorem}\label{thm2:r=2}
There exists a constant $C>0$ such that, a.a.s.,
\[
|Y-n/2|\le C\sqrt{n\log n}.
\]
\end{theorem}
\noindent (This has been proved in \cite[Theorem 4.7]{Scheinerman1988} with an unspecified term $o(n)$ instead of  our $\sqrt{n\log n}$.)

The problem can be generalized to $r$-matchings as follows. Imagine a distant planet inhabited by multi-leg creatures which for that reason wear not pairs but $r$-packs of socks, for some fixed $r\ge2$. Then the process of placing them in a drawer differs only in that socks are taken up from the floor when the last one from a pack is pulled from the sack. This corresponds to a random $r$-matching $\rm^{(r)}(n)$ drawn uniformly at random from all $(rn)!/(r!^nn!)$ $r$-matchings (i.e., partitions into $r$-element subsets) of the set $[rn]$. Again, $r$-matchings can be represented as $r$-fold Gauss words, where each letter appears exactly $r$ times and permuting the letters has no effect.

For instance, let $r=3$ and $n=12$ and look at the 4-edge 3-matching encoded by triples of distinct letters as $AABCDDDCBCBA$. Here the process goes as follows:
$$A,AA,AAB, AABC, AABCD, AABCDD, AABC, AABCC, AABCCB, AABB, AA,\emptyset,$$
so the maximum achieved (twice) is 6, while the average length is $42/12=7/2$.

Here again we may express the problem in terms of bipartite sub-matchings with edges, however, bearing weights corresponding to the number of vertices ``on the left''.
Formally, given an $r$-matching $M$ on $[rn]$, for each $k\in [rn]$, let $M_k$ be the set of all edges of $M$ with nonempty intersections with both $[k]$ and $[rn]\setminus[k]$. Further, let $x_k(M)=\sum_{e\in M_k}|e\cap[k]|$. Note that $x_k(M)$ represents exactly the number of socks on the floor after $k$ socks have been pulled from the sack. For instance, in the above example, $M_8$ consists of three edges, marked by letters $A,B$, and $C$, and $x_8(M)=2+1+2=5$.

Let $\rm^{(r)}(n)$ be a random $r$-matching on $[rn]$ and define random variables
$$X^{(r)}_k=x_k(\rm^{(r)}(n))\;,\quad\widebar{X^{(r)}}=\sum_kX^{(r)}_k/(rn)\;,\quad\mbox{and}\quad Y^{(r)}=\max_kX^{(r)}_k.$$
We generalize Theorems \ref{thm1:r=2} and \ref{thm2:r=2} to $r$-matchings as follows.
\begin{theorem}\label{thm1:general}
We have $\E\widebar{X^{(r)}}=\frac{(r-1)(rn+1)}{2(r+1)}$ and, for every $\omega(n)\to\infty$, a.a.s.
\[
\Big|\widebar{X^{(r)}}-\frac{(r-1)rn}{2(r+1)}\Big|\le \omega(n)\sqrt{n}.
\]
\end{theorem}

\begin{theorem}\label{thm2:general}
For every $r\ge2$, there exists a constant $C>0$ such that, a.a.s.
\[
\Big|Y^{(r)}-\frac{(r-1)n}{r^{\frac1{r-1}}}\Big|\le C\sqrt{n\log n}.
\]
\end{theorem}
\noindent In particular, for $r=3$, we infer that a.a.s.\ $\widebar{X^{(3)}}\sim\frac34n$ and $Y^{(3)}\sim\frac{2\sqrt{3}}{3}n$.

It is worth noting that Theorems \ref{thm1:r=2} and \ref{thm1:general} can be reformulated in terms of the distribution of the lengths of edges in a random matching. Indeed, given an edge $e=\{i,j\}$, where $1\le i<j\le 2n$, let $\ell(e):=j-i$ be the \emph{length} of $e$. Then, by a standard double counting argument, for every matching $M$ of $[2n]$,
$$\sum_{k=1}^{2n}x_k(M)=\sum_{e\in M}\ell(e).$$
Defining random variables $L:=\sum_{e\in\rm(n)}\ell(e)$ and $\bar L:=\frac Ln$, we have $\bar L=2\bar X$ and thus may restate Theorem \ref{thm1:r=2} as follows.

\begin{cor}\label{thm1:r=2,L}
We have $\E\bar L=\frac{2n+1}{3}$ and, for every $\omega(n)\to\infty$, a.a.s.
\[
| \bar L- 2n/3|\le \omega(n)\sqrt n.
\]
\end{cor}

This connection can be generalized to arbitrary $r\ge2$, provided one comes up with a suitable definition of the length of an $r$-element subset of $[rn]$ reflecting the nature of the socks problem, namely that  the smaller vertices get higher weights (as they stay longer on the floor). To this end, given $e=\{i_1<i_2<\cdots<i_r\}\in \rm^{(r)}(n)$, let
$$\ell(e):=(i_2-i_1)+2(i_3-i_2)+\cdots+(r-1)(i_r-i_{r-1})=(i_r-i_1)+(i_r-i_2)+\cdots(i_r-i_{r-1}).$$
It is not difficult to see that, for every $r$-matching $M$ of $[rn]$,
$$\sum_{k=1}^{rn}x_k(M)=\sum_{e\in M}\ell(e).$$
Indeed, vertex $i_j\in e$, $1\le j\le r$, contributes 1 to the numbers $x_{i_j},x_{i_j+1},\dots, x_{i_r-1}$ and~0 to all others, so its contribution equals $i_r-i_j$ and thus the total contribution of $e$ toward $\sum_{k=1}^{rn}x_k(M)$ is precisely $\ell(e)$.

Defining random variables $L^{(r)}:=\sum_{e\in\rm^{(r)}(n)}\ell(e)$ and $\bar L^{(r)}:=\frac{L^{(r)}}n$, we have $\bar L^{(r)}=r\bar X^{(r)}$ and  thus may restate Theorem \ref{thm1:general} as follows.

\begin{cor}\label{thm1:general,L}
We have $\E\bar L^{(r)}=\frac{r(r-1)(rn+1)}{2(r+1)}$ and, for every $\omega(n)\to\infty$, a.a.s.
\[
\Big| \bar L^{(r)}- \frac{(r-1)r^2n}{2(r+1)}\Big|\le \omega(n)\sqrt n.
\]
\end{cor}

\section{Proofs}
In this section we prove Theorems~\ref{thm1:general} and~\ref{thm2:general}. Recall that $X^{(r)}_k=\sum_{e\in{\rm_k^{(r)}(n)}}|e\cap[k]|$, where  $\rm_k^{(r)}(n)$ the set of all edges of $\rm^{(r)}(n)$ with nonempty intersections with both $[k]$ and $[rn]\setminus[k]$. So we may assume that $1\le k\le rn-1$.

 Both proofs rely on high concentration of respective random variables around their means: $\widebar{X^{(r)}}$ for Theorem~\ref{thm1:general}, the $X^{(r)}_k$'s for Theorem~\ref{thm2:general}. To this end, we will use the Azuma-Hoeffding inequality for random permutations (see, e.g., Lemma 11 in~\cite{FP} or  Section 3.2 in~\cite{McDiarmid98}). This is feasible because ordered matchings can be produced by generating random permutations. Indeed, let $\pi$ be a permutation of $[rn]$. It can be chopped off into an $r$-matching $M_\pi:=\{\pi(1)\dots\pi(r), \pi(r+1)\dots\pi(2r),\dots,\pi(rn-r+1)\dots\pi(rn)\}$ and, clearly, there are exactly $(r!)^n n!$ permutations $\pi$ yielding the same matching. Thus, we can use the following lemma. By swapping two elements in a permutation $\pi_1$ we mean fixing two indices $i<j$ and creating a new permutation $\pi_2$ with $\pi_2(i)=\pi_1(j)$, $\pi_2(j)=\pi_1(i)$, and $\pi_2(\ell)=\pi_1(\ell)$ for all $\ell\neq i,j$. Let $\Pi_{N}$ denote a permutation selected uniformly at random from all $N!$ permutations of $[N]$.

\begin{lemma}[\cite{FP} or \cite{McDiarmid98}]\label{azuma}
 Let $h(\pi)$ be a function defined on the set of all permutations of order $N$ which satisfies the following Lipschitz-type condition: there exists a constant $c>0$ such that whenever a permutation $\pi_2$ is obtained from a permutation $\pi_1$ by swapping two elements, we have $|h(\pi_1)-h(\pi_2)|\le c$.
Then, for every $\eta>0$,
\[
\PP(|h(\Pi_{N})-\E[h(\Pi_{N})]|\ge \eta)\le 2\exp\{-2\eta^2/(c^2N)\}.
\]
\end{lemma}

With respect to the Lipschitz condition the following observation is crucial for us.
Let $\pi_1$ be a permutation of $[rn]$ and permutation $\pi_2$ be obtained from $\pi_1$ by swapping two elements. This way we can destroy (or create) at most two edges which may jointly contribute at most $2(r-1)$ to $X^{(r)}_k$. Hence,   for each $k$ the random variable $X^{(r)}_k$ satisfies the Lipschitz condition with $c=2(r-1)$ and so does $\widebar{X^{(r)}}$.

Equipped with the above concentration tool, the next thing we need is a formula for $\E(X^{(r)}_k)$.

\begin{claim}\label{claim:1}
For all $1\le k\le rn-1$,
\[
\E(X^{(r)}_k) = k - \frac{r\binom{k}{r}}{\binom{rn-1}{r-1}}.
\]
\end{claim}

\begin{proof}[Proof of Claim~\ref{claim:1}]
In order to find the expected value of $X^{(r)}_k$,  fix an integer $1\le j\le r-1$ and observe that there are $\binom{k}{j}\binom{rn-k}{r-j}$  possible edges having exactly $j$ vertices in $[k]$. Since a fixed edge appears in a randomly chosen matching with probability $1/\binom{rn-1}{r-1}$, we get
\[
\E(X^{(r)}_k) = \sum_{j=1}^{r-1} \frac{j\binom{k}{j}\binom{rn-k}{r-j}}{\binom{rn-1}{r-1}}=\frac{k\sum_{j=0}^{r-1}\binom{k-1}{j}\binom{rn-k}{r-1-j}-k\binom{k-1}{r-1}}{\binom{rn-1}{r-1}}=k - \frac{r\binom{k}{r}}{\binom{rn-1}{r-1}}.
\]
\end{proof}

\begin{proof}[Proof of Theorem~\ref{thm1:general}]
By Claim \ref{claim:1} and the linearity of expectation
\begin{align*}
\E(\widebar{X^{(r)}}) &= \frac{1}{rn} \sum_{k=1}^{rn-1} \E(X^{(r)}_k)
= \frac{1}{rn} \sum_{k=1}^{rn-1} \left( k - \frac{r\binom{k}{r}}{\binom{rn-1}{r-1}} \right)
= \frac{rn-1}{2} - \frac{1}{n \binom{rn-1}{r-1}} \sum_{k=1}^{rn-1} \binom{k}{r}\\
&= \frac{rn-1}{2} - \frac{1}{n \binom{rn-1}{r-1}} \binom{rn}{r+1}
= \frac{rn-1}{2} - \frac{rn-r}{r+1}
= \frac{(r-1)(rn+1)}{2(r+1)}.
\end{align*}

 Recall that   $\widebar{X^{(r)}}$ satisfies the Lipschitz condition with $c=2(r-1)$.  Hence,  given $\omega(n)\to\infty$, Lemma~\ref{azuma} applied with $N=rn$, $h(\pi)=\sum_{k=1}^{rn}x_k(M_\pi)/(rn)$,   $c=2(r-1)$, and $\eta = \tfrac12\omega(n)\sqrt n$ yields (for large $n$)
 \begin{align*}
 \PP\left(\Big|\widebar{X^{(r)}}-\frac{(r-1)rn}{2(r+1)}\Big|\ge\omega(n)\sqrt n\right)&\le\PP\left(\Big|\widebar{X^{(r)}}-\E\widebar{X^{(r)}}\Big|\ge\frac12\omega(n)\sqrt n\right)\\&\le2\exp\left\{-\frac{\omega^2(n)}{2c^2r}\right\}=o(1).
 \end{align*}
\end{proof}
\begin{proof}[Proof of Theorem~\ref{thm2:general}] This proof is more tricky, as we do not have an exact formula for $\E Y$, where recall $Y=\max_kX^{(r)}_k.$ Instead, we apply Lemma~\ref{azuma} to each $X^{(r)}_k$ individually for a broad range of $k$.
First observe that $\E(X^{(r)}_k)$, already computed in Claim~\ref{claim:1}, can be alternatively expressed as
\begin{align*}
\E(X^{(r)}_k) &= k - \frac{r\binom{k}{r}}{\binom{rn-1}{r-1}}
= k - \frac{k\binom{k-1}{r-1}}{\binom{rn-1}{r-1}}
= k\left(1 - \frac{\binom{k-1}{r-1}}{\binom{rn-1}{r-1}} \right)\\
&= k\left(1 - \frac{k-1}{rn-1}\cdot \frac{k-2}{rn-2}\cdot \ldots \cdot \frac{k-r+1}{rn-r+1} \right)
\end{align*}
and thus
\begin{equation}\label{bounds}
k\left(1 - \left(\frac{k}{rn}\right)^{r-1} \right)
\le \E(X^{(r)}_k)
\le k\left(1 - \left(\frac{k-r}{rn-r}\right)^{r-1} \right).
\end{equation}
Also note that
\begin{align*}
\left(\frac{k-r}{rn-r}\right)^{r-1}
&\ge \left(\frac{k-r}{rn}\right)^{r-1}
= \left(\frac{k}{rn}-\frac{1}{n}\right)^{r-1}
= \left(\frac{k}{rn}\right)^{r-1} - O\left(\frac{k^{r-2}}{n^{r-1}}\right)
= \left(\frac{k}{rn}\right)^{r-1} - O\left(\frac{1}{k}\right),
\end{align*}
as $k\le rn$.
Hence,
\[
\E(X^{(r)}_k)
= k\left(1 - \left(\frac{k}{rn}\right)^{r-1} \right) + O(1).
\]

 Since $f(x)=x - \frac{x^{r}}{(rn)^{r-1}}$ for $1\le x\le rn$ achieves its maximum at $x_0 = r^{\frac{r-2}{r-1}}n$ and $f(x_0)=\frac{(r-1)n}{r^{1/(r-1)}}$,
it follows that for each $k$,
\begin{equation}\label{boundX}
\E(X^{(r)}_k)\le f(k_0)=\frac{(r-1)n}{r^{1/(r-1)}} + O(1).
\end{equation}
where $k_0=\lceil r^{\frac{r-2}{r-1}}n \rceil$. Moreover, by the L-H-S of~\eqref{bounds},
\begin{align*}
\E(X^{(r)}_{k_0})
&\ge k_0\left(1 - \left(\frac{k_0}{rn}\right)^{r-1} \right)
\ge r^{\frac{r-2}{r-1}}n \left(1 - \left(\frac{r^{\frac{r-2}{r-1}}n+1}{rn}\right)^{r-1} \right)\\
&= r^{\frac{r-2}{r-1}}n \left(1 - \left(\frac{1}{r^{1/(r-1)}} + \frac{1}{rn}\right)^{r-1} \right)
= \frac{(r-1)n}{r^{1/(r-1)}} + O(1).
\end{align*}

Let
\begin{equation}\label{eq:rangek}
n^{4/5}\le k \le rn - n^{4/5}.
\end{equation}
Observe that for any $k$ outside the above range,
\[
X_k^{(r)}\le\min\{k,(rn-k)(r-1)\}\le (r-1)n^{4/5}
\]
We next show that for each $k$ satisfying \eqref{eq:rangek} the random variable $X^{(r)}_k$ is highly concentrated around its mean. Indeed, for  $k$ in this range, by the L-H-S of \eqref{bounds},
\[
\E X^{(r)}_k =k - \frac{r\binom{k}{r}}{\binom{rn-1}{r-1}}
\ge k\left(1 - \frac{k}{rn}  \right)
\ge n^{4/5} \left(1 - \frac{rn-n^{4/5}}{rn} \right)
= \frac{n^{3/5}}{r}.
\]
 Thus, Lemma~\ref{azuma}, applied to $h(\pi)=x_k(M_\pi)$ with $N=rn$, $c=2(r-1)$, and $\eta = \tfrac12C\sqrt{n\log n}$  implies that for each $k$ satisfying~\eqref{eq:rangek}
 \begin{equation}\label{Az}
 \PP\left( |X^{(r)}_k - \E(X^{(r)}_k)| \ge \frac12C\sqrt{n\log n}\right)\le2\exp\left\{-\frac{C^2\log n}{2c^2r} \right\}=o(n^{-1}),
 \end{equation}
 provided $C>2(r-1)\sqrt{2r}$.

 Let us set $R=\frac{(r-1)}{r^{1/(r-1)}}$ and keep notation $\eta=\tfrac12C\sqrt{n\log n}$ for convenience. Recalling that $\E(X^{(r)}_{k_0})\ge Rn+O(1)$, it follows that a.a.s.\
 \[
 Y^{(r)}=\max_kX^{(r)}_k\ge X^{(r)}_{k_0}\ge \E X^{(r)}_{k_0}-\eta\ge Rn-2\eta.
 \]
 On the other hand, by \eqref{Az}, \eqref{boundX}, and the union bound,

\[
\PP\left(Y>Rn+2\eta\right)\le\PP\left(\exists k: X^{(r)}_k>Rn+2\eta\right)\le\sum_k\PP\left( |X^{(r)}_k - \E(X^{(r)}_k)| \ge \eta\right)=o(1),
\]
where the last sum is over all $k$ in the range \eqref{eq:rangek}. This completes the proof of Theorem~\ref{thm2:general}.
%{ n^{4/5}\le k \le rn - n^{4/5}}
\end{proof}

\section{Enumeration of matchings with a given sock number}\label{enu}
Let $\cM(n)$ be the set of all matchings of $[2n]$. For $M\in\cM(n)$ define the \emph{sock number} of $M$, $y(M):=\max_k x_k(M)$, to be the size of the largest bipartite sub-matching of $M$ (a particular instance of the random variable $Y$ defined earlier). For $1\le j\le n$, define
$$\cM_{\le j}(n)=\{M\in\cM(n): y(M)\le j\},$$
$$s_{\le j}(n)=|\cM_{\le j}(n)|\quad\mbox{and}\quad s_j(n)=s_{\le j}(n)-s_{\le j-1}(n),$$
where we assume that $s_{\le 0}(n)=0$.
 Thus, $s_{j}(n)$ counts the number of matchings of size $n$ with sock number exactly $j$.

It seems to be a very challenging problem to determine the values of $s_{j}(n)$. Here we just take the first step in this  direction by finding  numbers $s_{j}(n)$ for $j\le 2$ and $j\ge n-1$.
As a technical tool we introduce the \emph{Dyck trace} of $M$, that is a binary sequence $tr(M)=t_1\dots t_{2n}$, $t_i\in\{1,-1\}$, where $t_1=1$ if $i$ is the left end of an edge of $M$ and $t_i=-1$ if $i$ is  the right end of an edge of $M$. For instance, if $M=ABCCDBDA$, then $tr(M)=(1,1,1,-1,1,-1,-1,-1)$.
The name comes from \emph{Dyck sequences} which are binary sequences, with say $n$ 1's and $n$ $(-1)$'s, where each prefix has at least as many 1's as $(-1)$'s, equivalently, for each $1\le k\le2n$, we have $\sum_{i=1}^kt_i\ge0$.
Dyck sequences are enumerated by  the Catalan numbers $\frac1{n+1}\binom{2n}n$.

Clearly, several matchings may have the same Dyck trace. More precisely, define a \emph{run} in a sequence as a maximal block of identical elements. If the consecutive runs of 1's and -1's in a Dyck sequence have lengths $l_1,r_1,l_2,r_2,\dots$, then there are exactly $(l_1)_{r_1}\cdot(l_1+l_2-r_1)_{r_2}\cdot\cdots$ matchings with this Dyck trace, where $(l)_r=l(l-1)\cdots(l-r+1)$ is the falling factorial. (In the above
example, $l_1=3, r_1=1,l_2=1,r_2=3$, so there are $(3)_1\cdot(3)_3=3\cdot 3\cdot2\cdot1=18$ matchings with the same Dyck trace as $M$.)
Among them there is exactly one crossing-free and exactly one nesting-free matching which, by the way, proves that both these subfamilies of ordered matchings are enumerated by the Catalan numbers (see \cite{StanleyCatalan}).

Let us note that the sock number of $M$, $y(M)$, equals the \emph{height} of its trace $tr(M)$, that is, $\max_k\sum_{i=1}^kt_i$. Dyck sequences with height at most $j$ have been enumerated by a recurrence relation (see \url{https://oeis.org/A080934}) but no closed formula was found. It seems that determining $s_j(n)$ is not an easier task.

\begin{remark}
\normalfont
As  mentioned above, ordered matchings $M$ of $[2n]$ with no nestings are uniquely determined by their Dyck traces $tr(M)$, and the same is true for matchings with no crossings. Thus, the random matching $\rm(n)$ conditioned on containing no nestings, as well as the one conditioned on containing no crossings, is equivalent to a random Dyck sequence $\rd(n)$ of length $2n$. Thus, in these conditional spaces, the random variable $y(\rm(n))$ has the same distribution as the hight $h(\rd(n))$. As for the latter,
it can be shown by a standard application of Chernoff's inequality (see  Ineq. (2.9) together with Theorem 2.10 in \cite{JLR})  that a.a.s., $h(\rd(n))=O(\sqrt{n\log n})$, a significant drop from $\Theta(n)$ established in Theorem \ref{thm2:r=2} for unconditional $\rm(n)$.
\end{remark}

Trivially, $s_1(n)=1$, as $M_1=A_1A_1A_2A_2\cdots A_nA_n$ is the only matching with $y(M)=1$, and $s_n(n)=n!$, as $y(M)=n$ if and only if $M$ is bipartite, and so $M$ is determined by the order of the right endpoints of its edges. Our two final results provide formulae for $s_2(n)$ and $s_{n-1}(n)$, respectively.

\begin{prop}\label{k=2} We have  $s_{\le2}(1)=1$ and, for each $n\ge2$,
\[
s_{\le2}(n)=3s_{\le2}(n-1).
\]
Consequently, $s_{\le2}(n)=3^{n-1}$ and $s_2(n)=s_{\le2}(n)-s_1(n)=3^{n-1}-1$.
\end{prop}
\begin{proof}
 Note that for every $M\in\cM_{\le 2}(n)$ other than $M_1$, the longest run in $tr(M)$ has length two.
%There is just one $M$ where $tr(M)$ does not have any such run, namely $(1,-1,1,-1,\dots,1,-1)$. 
Let $\cM_{\le2}(n,t)$ be the set of all matchings in $M\in\cM_{\le2}(n)$ for which in $tr(M)$  the first (from the left) $11$-run occurs at positions $2t-1,2t$. (Note that any such run has to begin at an odd position.)
 The digit $-1$ which must follow this $11$ in  $tr(M)$ may represent the right end of an edge whose left end is either $2t-1$ or $2t$. Once this is decided, we are looking at a fresh sub-matching of size $n-t$. Thus,
 \[
s_{\le2}(n)= |\cM_{\le 2}(n)|=1+\sum_{t=1}^n|\cM_{\le2}(n,t)|=1+2\sum_{t=1}^n|\cM_{\le2}(n-t)|,
 \]
and the stated recurrence follows by canceling all common terms in $s_{\le2}(n)-s_{\le2}(n-1)$.
 \end{proof}

\begin{prop}\label{k=n-1} For each $n\ge1$,
\[
s_{n-1}(n)=(n-1)^2(n-1)!.
\]
\end{prop}
\begin{proof}
We have  $y(M)=n-1$ if and only if there is exactly one edge  $e\in M$ with both ends in $[n]$, that is, $e\subset[n]$. If $e\subset[n-1]$, then the $n-1$ elements of $[n+1]\setminus e$ can be matched arbitrarily with those in $[2n]\setminus[n+1]$. So, there are $\binom{n-1}2\times (n-1)!$ such matchings $M$. If $e=\{i,n\}$ for some $i\in[n-1]$, then the $n-2$ vertices in $[n-1]\setminus\{i\}$ can be matched arbitrarily with some $n-1$ elements of $[2n]\setminus[n]$, leaving out two elements to form the final edge of $M$. So, in this case there are $(n-1)\times(n)_{n-2}$ such matchings $M$.
Summing up the two quantities yields the required formula.
\end{proof}

\begin{prob}
	Determine $s_k(n)$ (by finding a compact formula,  a recurrence, or a generating function)  for all values of $n$ and $k$.
\end{prob}
\noindent It seems that even the exact determination of $s_3(n)$ may be quite challenging.

\section{Final remarks}

We conclude the paper with a number of far-reaching generalization of Bosek's socks problem.  We begin with a couple of problems which lack the sequential aspect of Bosek's questions, but can be viewed as their Tur\'an-type counterparts. Recall the dual nature of the random variable $Y$ defined in the Introduction which on the one hand is the maximum value of a sockuence, while on the other hand, it is the size of the largest bipartite sub-matching, equivalently, the largest sub-matching not containing pattern $AABB$.

  Two words over a finite alphabet are deemed to be \emph{isomorphic} if they are the same up to a permutation of the letters.   For instance, the words $ABBA$ and $DAAD$ are isomorphic. (Note that for $r$-fold Gauss words this is equivalent to the order preserving isomorphism of the corresponding ordered $r$-matchings.)

Let $\mathcal F$ be a fixed set of forbidden ordered $r$-matchings (or its corresponding Gauss words). An ordered $r$-matching $M$ is \emph{$\mathcal F$-free} if $M$ does not contain a sub-matching (order) isomorphic to $F$ for any $F\in\mathcal F$. What is the maximum size of an $\mathcal F$-free sub-matching of $\rm^{(r)}(n)$? As mentioned above, for $r=2$, Bosek's original question corresponds to the case of $\mathcal F=\{AABB\}$. Let us denote by $ex_{\mathcal F}(M)$ the maximum number of edges in an $\mathcal F$-free sub-matching of~$M$.

\begin{prob}
	For $r\ge2$, let $\mathcal F$ be a set of forbidden ordered $r$-matchings. Estimate the expected value of $ex_{\mathcal F}(\rm^{(r)}(n))$.
\end{prob}

Next, consider a special case of the above problem in which all elements of $\mathcal F$ are patterns, that is, pairs of edges. Given $r\ge2$, there are exactly $\tfrac12\binom{2r}r$ ways, called \emph{$r$-patterns}, in which a pair of disjoint edges of order $r$ may intertwine as ordered vertex sets. In this restricted setting, it is equivalent, and perhaps more natural, to emphasize not the forbidden  patterns but the complementary set of those which are allowed to be present. For a set of $r$-patterns~$\mathcal P$, a \emph{$\mathcal P$-clique} is defined as an $r$-matching whose all pairs of edges form patterns $P$ belonging to $\mathcal P$. Let us denote by $z_{\mathcal P}(M)$ the maximum number of edges in a $\mathcal P$-clique contained in $M$.

\begin{prob}\label{free} For $r\ge2$, let $\mathcal P$ be a set of $r$-patterns. Estimate the expected value of  $z_{{\mathcal P}}(\rm^{(r)}(n))$.
\end{prob}

Recall that, for $r=2$, a bipartite matching is one without alignments $AABB$, and so the above problem in that case, that is, when $\mathcal P=\{ABAB,ABBA\}$ is solved by Theorem \ref{thm2:r=2}.
On the other hand, it can be easily shown by the first moment method that if one instead forbids nestings or crossings, that is, if $\mathcal P=\{AABB,ABBA\}$ or $\mathcal P=\{AABB,ABAB\}$, then the answer is just $O(\sqrt n)$ (and not $\Theta(n)$) -- the same when $|\mathcal P|=1$ (see \cite{DGR-match}).

 For larger $r$, we know a typical order of magnitude of $z_{{\mathcal P}}(\rm^{(r)}(n))$ when $\mathcal P=\{P\}$. (In such case we use simplified notation $z_P(\rm^{(r)}(n))$.) To state the result, we have to distinguish a special subfamily of patterns. Call an $r$-pattern $P$ \emph{collectable} if for each $k\ge1$ there exists a $\{P\}$-clique of size $k$. For $r\ge3$ not every $r$-pattern is collectable.
For instance, the 3-pattern  $P_0=AABABB$ is the unique not collectable 3-pattern (it fails already at $k=3$). In fact, as shown in \cite{JCTB_paper}, there are exactly $3^{r-1}$ collectable $r$-patterns.

By definition,  for every non-collectable pattern $P$, we have $z_P(\rm^{(r)}(n))=O(1)$.
On the other hand, it was proved in \cite{JCTB_paper} that for every collectable $r$-pattern $P$ a.a.s.\ $z_P(\rm^{(r)}(n))=\Theta(n^{1/r})$. (In \cite{AJKS} the constants in front of $n^{1/r}$ have been, in principle, determined.)

Let us further restrict to $r$-partite patterns and sub-matchings (i.e., those whose interval chromatic number is $r$). There are exactly $2^{r-1}$ $r$-partite $r$-patterns and they are all collectable. Let ${\mathcal Q}^{(r)}$ be the family of all of them. In particular,
$${\mathcal Q}^{(3)}=\{ABABAB, ABBAAB, ABABBA, ABBABA\}.$$

\begin{prob}\label{tri} Determine a.a.s.\  the largest size of an $r$-partite sub-matching of the random $r$-matching
$\rm^{(r)}(n)$, that is, determine $z_{{\mathcal Q}^{(r)}}(\rm^{(r)}(n))$.
Further, determine a.a.s.\ the order of magnitude of  $z_{\mathcal P}(\rm^{(r)}(n))$, when $\mathcal P\subset {\mathcal Q}^{(r)}$ is a subset of the set of $r$-partite patterns.
\end{prob}

We believe that $z_{{\mathcal Q}^{(r)}}(\rm^{(r)}(n))=\Theta(n)$ but would like to pinpoint the multiplicative constant -- as in Theorem~\ref{thm2:general}.
For $r=3$ we further conjecture that whenever $\mathcal P\subset\mathcal Q^{(3)}$ and $|\mathcal P|=2$, a.a.s.\  $z_{{\mathcal P}}(\rm^{(3)}(n))=\Theta(n^{1/2})$.

\medskip

Finally, let us pose a  very general problem which does reflect the sequential aspect of Bosek's question. Here we consider general words over finite alphabets, not just Gauss words.
 Suppose that $\mathcal J$ is a finite family of words (to be \emph{persecuted}). For a given word $U$, let $U/{\mathcal J}$ be the (scattered) sub-word of $U$ obtained by deleting successively (from left to right) one isomorphic copy of an element $J\in{\mathcal J}$ at a time. More precisely, we scan the letters one by one and once we see such a copy emerging, we remove all letters forming it. If  more than one copy of the elements in $\mathcal J$ has been revealed at that moment, we delete, say, the one which is farther to the left (destroying, in fact, all of them).
 
  For instance, if $\mathcal J=\{ABA\}$ and $U=ABCADBCDA$, then $U/{\mathcal J}=BDA$. Indeed, the first removal (of $ABA$, and not of $ACA$) takes place after four letters have been scanned and leaves the sub-word $CDBCDA$ which we continue to scan from the left; the second removal (of $CDC$, and not of $CBC$) leaves the final sub-word $BDA$.

  Let $y_{{\mathcal J}}(U)=\max_i|U_i/{\mathcal J}|$, where $U_i$ is the \emph{prefix} of $U$ of length $i$.
  (In the above example, $y_{{\mathcal J}}(U)=3$.) Notice that when $U$ is a Gauss word and $\mathcal J=\{AA\}$,  the number $y_{{\mathcal J}}(U)$ coincides with the sock number $y(M)$ of the corresponding ordered matching $M$.

Let $\rw_k(n)$ be any model of a random word of length $n$ over an alphabet of size $k$. (E.g., in \cite{DGR_annals}, two natural uniform models are defined).

  \begin{prob}
	For $k\ge2$  and a fixed persecuted family $\mathcal J$, estimate the expected value of $y_{{\mathcal J}}(\rw_k(n))$.
\end{prob}
\noindent  Note that Bosek's  problem for $r$-sets of socks asks for $y_{\mathcal J}(\rw_k(n))$, where $\mathcal J=\{A^r\}$,   $k=n$, $n:=rn$, and $\rw_k(n)$ is a random $r$-fold Gauss word of length $rn$.

\subsection*{Acknowledgements}
We would like to thank Bart\l{}omiej Bosek for sharing his ``real-life'' problem.

The first author was supported in part by Simons Foundation Grant \#522400. The second author was supported in part by Narodowe Centrum Nauki, grant 2020/37/B/ST1/03298.

%\bibliographystyle{amsplain}
%\bibliography{refs}

\begin{bibdiv}
\begin{biblist}

\bib{AJKS}{unpublished}{
      author={Anastos, Michael},
      author={Jin, Zhihan},
      author={Kwan, Matthew},
      author={Sudakov, Benny},
       title={Extremal, enumerative and probabilistic results on ordered
  hypergraph matchings, manuscript},
        note={\href{https://arxiv.org/pdf/2308.12268.pdf}{arXiv:2308.12268}},
}

\bib{BaikRains}{article}{
      author={Baik, Jinho},
      author={Rains, Eric~M.},
       title={The asymptotics of monotone subsequences of involutions},
        date={2001},
        ISSN={0012-7094},
     journal={Duke Math. J.},
      volume={109},
      number={2},
       pages={205\ndash 281},
         url={https://doi.org/10.1215/S0012-7094-01-10921-6},
      review={\MR{1845180}},
}

\bib{Bosek}{misc}{
      author={Bosek, B.},
        note={Personal communication},
}

\bib{JCTB_paper}{unpublished}{
      author={Dudek, Andrzej},
      author={Grytczuk, Jaros{\l a}w},
      author={Ruci\'{n}ski, Andrzej},
       title={Erd{\H o}s-{S}zekeres type theorems for ordered uniform
  matchings},
        note={\href{https://arxiv.org/pdf/2301.02936.pdf}{arXiv:2301.02936}},
}

\bib{DGR_annals}{article}{
      author={Dudek, Andrzej},
      author={Grytczuk, Jaros{\l a}w},
      author={Ruci\'{n}ski, Andrzej},
       title={Long twins in random words},
        date={2023},
        ISSN={0218-0006,0219-3094},
     journal={Ann. Comb.},
      volume={27},
      number={3},
       pages={749\ndash 768},
         url={https://doi.org/10.1007/s00026-023-00651-5},
      review={\MR{4633764}},
}

\bib{DGR-match}{article}{
      author={Dudek, Andrzej},
      author={Grytczuk, Jaros{\l a}w},
      author={Ruci\'{n}ski, Andrzej},
       title={Ordered unavoidable sub-structures in matchings and random
  matchings},
        date={2024},
        ISSN={1077-8926},
     journal={Electron. J. Combin.},
      volume={31},
      number={2},
       pages={Paper No. 2.15},
         url={https://doi.org/10.37236/11932},
      review={\MR{4734453}},
}

\bib{FP}{incollection}{
      author={Frieze, Alan},
      author={Pittel, Boris},
       title={Perfect matchings in random graphs with prescribed minimal
  degree},
        date={2004},
   booktitle={Mathematics and computer science. {III}},
      series={Trends Math.},
   publisher={Birkh\"{a}user, Basel},
       pages={95\ndash 132},
      review={\MR{2090500}},
}

\bib{FJKMV}{article}{
      author={F\"{u}redi, Zolt\'{a}n},
      author={Jiang, Tao},
      author={Kostochka, Alexandr},
      author={Mubayi, Dhruv},
      author={Verstra\"{e}te, Jacques},
       title={Extremal problems for convex geometric hypergraphs and ordered
  hypergraphs},
        date={2021},
        ISSN={0008-414X},
     journal={Canad. J. Math.},
      volume={73},
      number={6},
       pages={1648\ndash 1666},
         url={https://doi.org/10.4153/S0008414X20000632},
      review={\MR{4350556}},
}

\bib{JLR}{book}{
      author={Janson, Svante},
      author={\L{}uczak, Tomasz},
      author={Ruci\'nski, Andrzej},
       title={Random graphs},
      series={Wiley-Interscience Series in Discrete Mathematics and
  Optimization},
   publisher={Wiley-Interscience, New York},
        date={2000},
        ISBN={0-471-17541-2},
         url={https://doi.org/10.1002/9781118032718},
      review={\MR{1782847}},
}

\bib{JSW}{article}{
      author={Justicz, Joyce},
      author={Scheinerman, Edward~R.},
      author={Winkler, Peter~M.},
       title={Random intervals},
        date={1990},
        ISSN={0002-9890,1930-0972},
     journal={Amer. Math. Monthly},
      volume={97},
      number={10},
       pages={881\ndash 889},
         url={https://doi.org/10.2307/2324324},
      review={\MR{1079974}},
}

\bib{McDiarmid98}{book}{
      author={McDiarmid, Colin},
      editor={Habib, M.},
      editor={McDiarmid, C.},
      editor={Ramirez-Alfonsin, J.},
      editor={Reed, B.},
       title={Probabilistic methods for algorithmic discrete mathematics},
      series={Algorithms and Combinatorics},
   publisher={Springer-Verlag, Berlin},
        date={1998},
      volume={16},
        ISBN={3-540-64622-1},
         url={https://doi.org/10.1007/978-3-662-12788-9},
      review={\MR{1678554}},
}

\bib{Scheinerman1988}{article}{
      author={Scheinerman, E.~R.},
       title={Random interval graphs},
        date={1988},
        ISSN={0209-9683},
     journal={Combinatorica},
      volume={8},
      number={4},
       pages={357\ndash 371},
         url={https://doi.org/10.1007/BF02189092},
      review={\MR{981893}},
}

\bib{StanleyCatalan}{book}{
      author={Stanley, Richard~P.},
       title={Catalan numbers},
   publisher={Cambridge University Press, New York},
        date={2015},
        ISBN={978-1-107-42774-7; 978-1-107-07509-2},
      review={\MR{3467982}},
}

\end{biblist}
\end{bibdiv}


\end{document}